\documentclass[francais,a4paper]{article}

\usepackage{amsmath,amsfonts,amssymb,amsthm}

\usepackage[francais,english]{babel}
\usepackage[T1]{fontenc}
\usepackage[latin1]{inputenc}
\usepackage{textcomp}
\usepackage{enumerate}
\usepackage{verbatim}

\usepackage{graphicx}
\usepackage{epsfig}
\usepackage{color}

\usepackage[all]{xy}

\newcommand{\ob}{\pi}

\newcommand{\R}{\mathbb{R}}
\newcommand{\CP}{{\mathbb{C}}P}
\newcommand{\RP}{\mathbb{R}P}
\newcommand{\RA}{\mathbb{R}A}
\newcommand{\RC}{\mathbb{R}C}
\newcommand{\Z}{\mathbb{Z}}

\newcommand{\RMor}{Mor_{\R}}

\newcommand{\OC}{\mathcal{O}_{\CP^1}}
\newcommand{\OP}{\mathcal{O}_{C}}
\newcommand{\Nu}{\mathcal{N}_u}
\newcommand{\Nud}{\mathcal{N}_{u_0}}
\newcommand{\Nua}{\mathcal{N}_{u_1}}
\newcommand{\Nui}{\mathcal{N}_{u,-\ob^{2d}_i(\underline{z})}}
\newcommand{\Nuid}{\mathcal{N}_{u_0,-\ob^{2d}_i(z^0)}}

\newcommand{\Nuidtd}{\mathcal{N}_{u_0,-\ob^{2d}_i(\tilde{\underline{z}}^0)}}

\newcommand{\ev}{ev}
\newcommand{\Rev}{ev_\mathbb{R}}
\newcommand{\evl}{ev}

\newcommand{\K}{\mathcal{K}^d}

\newcommand{\Kl}{\mathcal{K}_{\ell}^{d}}
\newcommand{\RKl}{\mathbb{R}\mathcal{K}_{\ell}^{d}}
\newcommand{\D}{\mathcal{D}^d_{d',\ell}}

\newcommand{\Bl}{\mathfrak{Bl}}

\newcommand{\Mc}{\overline{\mathcal{M}}_{2d}^d(\CP^3)}

\newcommand{\Mck}{\overline{\mathcal{M}}_{k}^d(\CP^3)}

\newcommand{\Mco}{\overline{\mathcal{M}}_{0}^d(\CP^3)}
\newcommand{\Mless}{\overline{\mathcal{M}}_{2d-1}^d(\CP^3)}

\newcommand{\TMc}{T_{\overline{\mathcal{M}}_{2d}^d(\CP^3)^*}}

\newcommand{\TMmc}{T_{\map}\overline{\mathcal{M}}_{2d}^d(\CP^3)}

\newcommand{\RM}{\mathbb{R}\mathcal{M}_{2d}^d(\CP^3)}
\newcommand{\RMc}{\mathbb{R}\overline{\mathcal{M}}_{2d}^d(\CP^3)}

\newcommand{\RMl}{\mathbb{R}\mathcal{M}_{2d+\ell}^d(\CP^3)}
\newcommand{\RMcl}{\mathbb{R}\overline{\mathcal{M}}_{2d+\ell}^d(\CP^3)}

\newcommand{\TRMc}{T_{\mathbb{R}\overline{\mathcal{M}}_{2d}^d(\CP^3)^*}}

\newcommand{\map}{(C,\underline{z},u)}
\newcommand{\mapr}{(C_{\star},\underline{z}^{\star},{u}_{\star})}
\newcommand{\mapd}{(C_0,\underline{z}^0,{u}_0)}
\newcommand{\mapdtd}{(C_0,\underline{\tilde{z}}^0,{u}_0)}
\newcommand{\mapa}{(C_1,\underline{z}^1,{u}_1)}
\newcommand{\mapatd}{(C_1,\underline{\tilde{z}}^1,{u}_1)}
\newcommand{\mapt}{(C_t,\underline{z}^t,{u}_t)}
\newcommand{\mapttd}{(C_t,\underline{\tilde{z}}^t,{u}_t)}
\newcommand{\maprtd}{(C_{\star},\underline{\tilde{z}}^{\star},{u}_{\star})}

\newcommand{\TRMmc}{T_{\map}\RMc}

\newcommand{\TRMd}{T_{\gamma(0)}\RM}
\newcommand{\TRMdtd}{T_{\tilde{\gamma}(0)}\RM}
\newcommand{\TRMa}{T_{\gamma(1)}\RM}

\newcommand{\basd}{(e_1^0,\dots,e_{2d}^0,f_1^0,g_1^0,\dots,f_{2d}^0,g_{2d}^0)}

\newcommand{\basa}{(e_1^1,\dots,e_{2d}^1,f_1^1,g_1^1,\dots,f_{2d}^1,g_{2d}^1)}

\newcommand{\bast}{(e_1^t,\dots,e_{2d}^t,\tilde{f}_1^t,\tilde{g}_1^t,\dots,\tilde{f}_{2d}^t,\tilde{g}_{2d}^t)}

\DeclareMathOperator{\Ker}{Ker}

\theoremstyle{plain}
\newtheorem{theo}{Théorème}
\newtheorem{lemm}{Lemme}

\newtheorem{prop}{Proposition}

\theoremstyle{definition}
\newtheorem{defi}{Définition}
\newtheorem*{nota}{Notation}

\theoremstyle{remark}
\newtheorem*{rema}{Remarque}

\begin{document}
\selectlanguage{francais}
\author{Nicolas Puignau\footnote{Pendant la préparation de cet article, l'auteur fut partiellement soutenu par une bourse de recherche de la DAAD à l'université Technique de Berlin en Allemagne.}\\IMPA\\Rio de Janeiro, Brésil}
\date{}
\title{Sur la première classe de Stiefel-Whitney de l'espace des applications stables réelles vers l'espace projectif}
\maketitle
\selectlanguage{english}
\begin{abstract}
Moduli space of genus zero stable maps to the projective three-space naturally carries a real structure such that the fixed locus is a moduli space for real rational spatial curves with real marked points. The latter is a normal projective real variety. The singular locus being in codimension at least two, a first Stiefel-Whitney class is well defined. In this paper, we determine a representative for the first Stiefel-Whitney class of such real space when the evaluation map is generically finite. This can be done by means of Poincaré duals of boundary divisors.
\end{abstract}
\selectlanguage{francais}

\setcounter{tocdepth}{3}
\setcounter{secnumdepth}{3}
\tableofcontents

\section{Introduction}
Nous nous intéressons aux courbes rationnelles réelles pointées de l'espace projectif $\CP^3$. Dans \cite{W1} Welschinger définit des invariants pour les variétés convexes de dimension trois équipées d'une structure réelles. Ces derniers sont déterminés via la géométrie énumérative des courbes rationnelles réelles de la variété. Pour un entier $d$ supérieur ou égal à trois et une configuration générique de $2d$ points dans l'espace complexe projectif $\CP^3$ les courbes rationnelles de degré $d$ qui passent par ces points sont en nombre fini. Ce nombre est indépendant de la configuration choisie parce que le corps des nombres complexes est algébriquement clos, c'est un invariant de Gromov-Witten de $\CP^3$. Le m\^eme problème sur le corps des nombres réels dépend de la configuration de points choisie. Dans ce contexte, un invariant de Welschinger est défini comme la somme algébrique sur les courbes réelles de la collection, chacune comptée avec un signe approprié qui ne dépend que de la courbe plongée. Cette somme est alors indépendante de la configuration de points générique (voir \cite{W1}).

Soit $\Mc$ l'espace des modules des applications stables en genre zéro (omis dans la notation) de degré $d$ dans $\CP^3$ avec $2d$ points marqués (voir \cite{F-P}). L'invariant de Gromov-Witten, décrit précédemment, n'est autre que le degré du morphisme d'évaluation $\ev : \Mc \to (\CP^3)^{2d}$ qui envoie une courbe marquée sur la configuration des marquages.
Il existe une contrepartie réelle de cet espace, notée $\RMc$, qui est une variété projective normale et un espace de paramètres pour les courbes spatiales rationnelles réelles de degré $d$ munies de $2d$ points marqués réels (voir \cite{W1}). L'invariant de Welschinger peut également s'interpréter comme le degré du morphisme d'évaluation $\Rev$ entre les espaces $\RMc$ et $(\RP^3)^{2d}$. En fait, ce degré n'est bien défini que parce que $\RMc$ peut-\^etre orienté en dehors du lieu où le jacobien du morphisme d'évaluation est de corang au moins deux. Cette propriété est au coeur de l'existence et de la définition des invariants de Welschinger. Il en résulte que la classe principale de $\RMc$ a été grossièrement caractérisée par Welschinger (voir \cite{W2}) dans une approximation suffisante pour construire ces invariants. Le travail actuel accomplit cette t\^ache en donnant une représentation précise de la première classe de Stiefel-Whitney de $\RMc$.

Pour $A\sqcup B$ une partition de $[2d]=\{1,\dots,2d\}$ et des entiers $d'$ et $d''$ tels que $d'+d''=d$ est défini un diviseur irréductible $D(A,B;d',d'')$ de $\Mc$ comme la fermeture du lieu des applications stables $\map$ telles que
\begin{enumerate}
\item $C=C_A\cup C_B$ est l'union de deux courbes $C_A$ et $C_B$ de genre zéro sécantes transversalement en un point unique ;
\item les marquages indexés par $A$ (resp. $B$) appartiennent à $C_A$ (resp. $C_B$) ;
\item $u(C_A)$ est de degré $d'$ et $u(C_B)$ de degré $d''$.
\end{enumerate}
La frontière de $\Mc$ est la réunion de tous ces diviseurs (voir \cite{F-P}).
Soit maintenant $\map$ une application stable de $\Mc$ telle que $C=\CP^1$ et $u$ soit une immersion. Le faisceau normal $\Nu$ défini comme le quotient de la séquence $\xymatrix{0 \ar[r] & T_C \ar[r]^-{du} & u^*T_{\CP^3} }$ est le faisceau d'un fibré vectoriel de rang deux sur $\CP^1$. Il est donc décomposable en somme directe de deux fibrés en droites. Lorsque les fibrés linéaires de la décomposition sont isomorphes, on dit que la courbe définie par $\map$ est \emph{équilibrée}. Notons $K^d \subset \Mc$ la réunion des diviseurs irréductibles de la frontière dont l'image par le morphisme d'évaluation est de codimension au moins deux et du lieu des courbes non équilibrées $\map$ telles que $h^1\left(C;\Nu \otimes \OP(-\underline{z})\right)$ soit au moins deux. L'image de la partie réelle $\R K^d$ est de codimension au moins deux dans $(\RP^3)^{2d}$.

Une orientation sur $\RMc$ est définie partout en dehors de $\R K^d$ ce qui permet d'identifier l'invariant de Welschinger et le degré du morphisme d'évaluation réel (\textit{cf}. proposition 4.2 de \cite{W2}). En conséquence, la première classe de Stiefel-Whitney de $\RMc$ admet un représentant dual inclus dans $\R K^d$ et peut donc s'écrire en termes de composantes pseudo-connexes de celui-ci (par composante pseudo-connexe on entend l'adhérence d'une composante connexe du complémentaire du lieu singulier dans le lieu réel de $K^d$). Nous déterminons exactement quelles sont les composantes pseudo-connexes de $\R K^d$ qui participent effectivement à représenter la première classe de Stiefel-Whitney de $\RMc$. Il advient que les composantes qui ne sont pas dans la frontière ne participent pas car elles sont de codimension trop grande (section \ref{reduc}). De plus, la participation d'une composante pseudo-connexe de la fontière réelle ne dépend que de la donnée du diviseur irréductible qui la contient. Pour une classe $[D]$ dans $H_{6d-1}(\RMc,\Z/2\Z)$, on note $[D]^{\vee}$ son image par le morphisme $H_{6d-1}(\RMc,\Z/2\Z) \to H^{1}(\RMc,\Z/2\Z)$. Cet article vise à démontrer le théorème suivant.

\begin{theo}\label{theo}
La première classe de Stiefel-Whitney de $\RMc$ s'écrit
$$w_1(\RMc)=\sum_{\substack{0 < \ell < 2d}}\ell.[\RKl]^{\vee}$$
où $\R\Kl$ est la partie réelle de la réunion des diviseurs $D(A,B,d',d'')$ tels que $|A|=2d'+\ell$.
\end{theo}

Dans la démonstration nous remarquons d'abord (section \ref{reduc}) que seules les composantes de la frontière sont de codimension assez petite pour représenter la première classe de Stiefel-Whitney. Ensuite nous procédons à une étude locale des orientations au voisinage d'un composante pseudo-connexe quelconque de $\R K^d$. Pour cela on transporte une orientation locale le long d'un chemin générique transverse à la composante et on compare cette orientation à celle issue d'un chemin de comparaison qui lui se situe dans le lieu régulier du morphisme d'évaluation (section \ref{trav}). Cette approche locale est possible du fait que $\RMc$ soit orienté en dehors de $\R K^d$ en relevant une orientation de $\RP^3$ par le morphisme d'évaluation réel suivant les signes de Welschinger. Donc pour comparer les orientations locales, on les compare successivement au tiré-en-arrière d'une orientation de référence sur $\RP^3$ (section \ref{evalbase}).\\

\noindent{\bf Remerciements.} Je remercie le relecteur de cet article pour m'avoir révéler un certain nombre d'inexactitudes contenues dans sa version préliminaire.

\section{Démonstration}

\subsection{Réduction à la frontière}\label{reduc}
Nous réduisons notre étude aux composantes pseudo-connexes contenues dans certains diviseurs irréductibles de la frontière.
\begin{prop}\label{reducfront}
La première classe de Stiefel-Whitney de $\RMc$ est représentable (dualement) en termes de composantes pseudo-connexes de la partie réelle des diviseurs de la frontière $D(A,B,d',d'')$ tels que $|A|\neq2d'$.
\end{prop}
\begin{proof}
La démonstration de cette proposition repose essentiellement sur les deux lemmes suivants.
\begin{lemm}[Eisenbud-Van de Ven \cite{E-V}]\label{EVdV}
Soit $Mor^d$ l'ensemble des morphismes $u: \CP^1 \to \CP^3$ de degré $d$. Pour $\rho$ tel que $1\leq \rho \leq d-4$ l'ensemble $S_{d,\rho} \subset Mor^d$, défini comme le lieu des morphismes tels que le fibré normal admet une décomposition du type $\OC(2d-1-\rho) \oplus \OC(2d-1+\rho)$, est une sous-variété de dimension $\dim  S_{d,\rho}=4d-2\rho+4$ (soit de codimension $2\rho-1$).
\end{lemm}
Ainsi les composantes de $K^d$ définies comme le lieu des applications stables $\map$ non équilibrées et telles que la dimension de $H^1\left(C;\Nu \otimes \OP(-\underline{z})\right)$ soit au moins deux, sont en codimension trop grande ($\geq2$) dans $\Mc$. On peut donc exclure de notre étude les composantes de la partie réelle de $K^d$ correspondant au lieu des courbes non équilibrées puisque la première classe de Stiefel-Whitney est un élément de $H^1(\RMc,\Z/2\Z)$.

Il est possible de raffiner ce résultat en excluant la partie réelle des diviseurs irréductibles de la frontière qui ne sont pas écrasés par le morphisme d'évaluation puisque celui-ci est génériquement régulier à cet endroit (voir \cite{W1}).
\begin{lemm}[\cite{F-P}]\label{imcodim1}
L'image d'un diviseur de la frontière $D(A,B;d',d'')$ par le morphisme d'évaluation $\ev : \Mc \to (\CP^3)^{2d}$ est de codimension un dans $(\CP^3)^{2d}$ si et seulement si $|A|=2d'$.
\end{lemm}
Ainsi, seuls les diviseurs irréductibles de la frontière $D(A,B;d',d'')$ qui vérifient $|A|\neq2d'$ contribuent éventuellement à représenter par dualité la première classe de Stiefel-Whitney de la partie réelle.
\end{proof}

\subsection{Bases de l'espace tangent}

\paragraph{Morphismes d'oubli.}\label{oubli}
Soient $l \leq k$ des entiers, il existe un morphisme entre variétés projectives de $\Mck$ vers $\overline{\mathcal{M}}_l^d(\CP^3)$ qui consiste à supprimer certains points marqués d'une application stable $\map \in \Mck$ puis à contracter les branches de la source $C$ devenues instables par insuffisance de points marqués (voir \cite{F-P}).

\begin{defi}Pour un ensemble d'indices $\{i_1,\dots,i_n\}$ dans $[k]$, on note $\ob^k_{i_1,\dots,i_n}$ le \emph{morphisme d'oubli} de $\Mck$ vers $\overline{\mathcal{M}}_{k-n}^d(\CP^3)$ qui supprime les points marqués indexés par l'ensemble $\{i_1,\dots,i_n\}$. Par exemple, pour $\{i\}$ dans $[k]$ $$\begin{array}{cccc}\ob^k_i :&\Mck &\to& \overline{\mathcal{M}}_{k-1}^d(\CP^3) \\& (C,z_1,\dots,z_k,u) &\mapsto& (C_{\ob_i^k}^{stab},z_1,\dots,\hat{z}_i,\dots,z_k,u).\end{array}$$
où $C_{\ob_i^k}^{stab}$ désigne la source éventuellement contractée $C \twoheadrightarrow C_{\ob_i^k}^{stab}$ pour stabiliser. Pour alléger les notations on écrit $\ob$ pour $\ob^k_{[k]}$ l'application qui oublie tous les points marqués.
\end{defi}

L'espace des modules contient un ouvert dense et lisse $\Mc^*$, lieu des applications stables sans automorphisme ou \emph{applications simples} (voir \cite{F-P}).

\begin{lemm}\label{propkero}
Soit $\map$ une application simple de $\Mc^*$, qui ne contracte pas sa source $C$, alors on a une décomposition du noyau $\ker d|_{\map}\ob = \bigoplus_{i=1}^k \ker d|_{\map}\ob^k_i$ et un isomorphisme $\ker d|_{\map}\ob \cong T_{\underline{z}}C^k$.
\end{lemm}

\begin{proof}
Soit $i$ dans $[2d]$, en l'absence d'automorphisme le morphisme $\ob_i^k$ décrit la courbe universelle au-dessus de l'image $\ob_i^k(\Mck^*)$ (voir \cite{F-P}). Les automorphismes d'une application stable étant un sous-groupe du groupe des automorphismes de son image par un morphisme d'oubli, $\overline{\mathcal{M}}_{k-1}^d(\CP^3)^*$ est inclus dans $\ob_i^k(\Mck^*)$. Donc $(\ob_i^k)^{-1}(C_{\ob_i^k}^{stab},z_1,\dots,\hat{z}_i,\dots,z_k,u)$ est isomorphisme à $C_{\ob_i^k}^{stab}$. De plus, par hypothèse sur $u$ on a l'égalité $C_{\ob_i^k}^{stab}=C$. Donc au point $(C,z_1,\dots,z_k,u)$, on a $\ker d|_{\map}\ob_i ^k= T_{z_i}C_{\ob_i^k}^{stab}=T_{z_i}C$. Pour tout $i,j$ dans $[k]$ distincts, on a l'égalité $\ob_j^{k-1}\circ\ob^k_i=\ob_i^{k-1}\circ\ob^k_j$ et donc $\ker d|_{\map}\ob= \bigoplus_{i=1}^{k} \ker d|_{\map}\ob_i^k$. On en déduit un isomorphisme entre $\ker d|_{\map}\ob$ et $T_{\underline{z}}C^k$.
\end{proof}
On note $\Ker d\ob$ le sous-ensemble du fibré tangent $\TMc$ défini par le noyau de la différentielle du morphisme d'oubli en chaque point. Sa restriction au lieu des applications stables simples non contractantes est un sous-fibré linéaire.

\subsubsection{Bases standards, bases modèles.}

\begin{prop}\label{decomptan}
Soit $\map$ une immersion de $\Mc$, on a la suite exacte
\begin{equation}\label{decompP}
\xymatrix{0 \ar[r] & T_zC \ar[r] & T_{\map}\Mc \ar[r]^-{d\ob} & H^0(C,\Nu) \ar[r] & 0}
\end{equation}
où $T_zC=\bigoplus_{i=1}^{k}T_{z_i}C$.
\end{prop}

\begin{proof}
La suite exacte se construit sur le morphisme d'oubli d'après le lemme \ref{propkero} avec $T_{(C,u)}\Mco$ au conoyau. Or, $T_{(C,u)}\Mco$ comme espace tangent est l'espace des déformations à l'ordre un de la courbe immergée définie par $u$, c'est-à-dire $H^0(C,\Nu)$.
\end{proof}

\begin{defi}\label{standP}
Une \emph{base standard} de $\TMmc$ est une base adaptée à la suite exacte \eqref{decompP}. C'est-à-dire, un $6d$-uplet composé de $2d$ éléments $\{e_i\}_{i=1,\dots,k}$, générateurs de chaque $T_{z_i}\CP^1$ et $4d$ éléments $\{f_i,g_i\}_{i=1,\dots,2d}$ de sorte que $(f_i,g_i)$ se projette sur une base de $H^0(\CP^1,\Nu)$. On écrira une telle base dans cet ordre $(e_1,\dots,e_{k},f_1,g_1,\dots,f_{2d},g_{2d})$.
\end{defi}

Pour $i$ dans $[2d]$, on note $T^{(i)}_{(\CP^3)^{2d}}$ le sous-fibré de $T_{(\CP^3)^{2d}}$ associé à la $i$-ème composante $\{\vec{0}\} ×\dots ×T_{\CP^3} ×\dots ×\{\vec{0}\}$. La décomposition de $T_{(\CP^3)^{2d}}=\bigoplus_{i=1}^{2d}T^{(i)}_{(\CP^3)^{2d}}$ se relève par le morphisme d'évaluation en tout point régulier \begin{equation}\label{decomp}T_{\bullet}{\Mc}=\bigoplus_{i=1}^{2d}\ev^{*}[T^{(i)}_{\ev(\bullet)}(\CP^3)^{2d}].\end{equation}

\begin{nota}Pour $\map$ une immersion et $i$ dans $[2d]$, le fibré $\Nu \otimes \mathcal{O}_{C}(-\hat{\underline{z}^i})$ où $\hat{\underline{z}}^i$ désigne $(z_1,\dots,\hat{z}_i,\dots,z_{2d}) \in C^{2d-1}$ sera noté $\Nui$.
\end{nota}
\begin{prop}\label{decompfib}
Soit $\map$ une immersion de $\Mc$ et $i$ dans $[2d]$, alors on a la suite exacte
\begin{displaymath}
\xymatrix{0 \ar[r] & T_{z_i}C \ar[r] & \ev^*T^{(i)}_{u(z_i)}(\CP^3)^{2d} \ar[r] & H^0(C,\Nui) \ar[r] & 0}
\end{displaymath}
\end{prop}

\begin{proof}
Pour $i$ dans $[2d]$, on définit $\hat{ev}_i$ le morphisme composé $$\hat{ev}_i:\Mc \xrightarrow{\ob^{2d}_i} \Mless \xrightarrow{\evl} (\CP^3)^{2d-1}$$ de sorte que le tiré-en-arrière $\ev^*(T^{(i)}_{(\CP^3)^{2d}})$ soit précisément $\Ker d\hat{ev}_i$. La suite exacte $\xymatrix{0 \ar[r] & \Ker d\ob^{2d}_i \ar[r] & \Ker d\hat{ev}_{i} \ar[r]^-{d\ob^{2d}_i} & d\ob^{2d}_i(\Ker d\hat{ev}_{i}) \ar[r] & 0}$ fournit l'égalité $d\ob^{2d}_i(\Ker d\hat{ev}_{i})=\Ker d\evl$ (voir \cite{Pu}). On conclut ensuite en appliquant les égalités $\ker d\ob^{2d}_i = T_{z_i}C$ (proposition \ref{decomptan}) et $\ker d\evl = H^0(C,\Nui)$ (lemme 1.3 de \cite{W1}]).
\end{proof}

\begin{defi}\label{modelP}
Soit $\map$ une immersion équilibrée : le fibré normal $\Nu$ admet une décomposition en somme directe de deux fibrés en droites isomorphes. Une base standard $\mathcal{B}$ de $T_{\map}\Mc$ est une \emph{base modèle} lorsqu'elle respecte la décomposition \eqref{decomp} et la décomposition du fibré normal. C'est-à-dire que chaque élément de la base appartient à un des sous-espaces $\ev^*T^{(i)}_{\map}(\CP^3)^{2d}$ et que les couples de vecteurs qui se projettent sur une base de $H^0(\CP^1,\Nui)$ forment une base compatible avec la décomposition du fibré normal (une section dans chaque sous-fibré linéaire). Si de plus $\mathcal{B}=(e_1,\dots,e_{2d},f_1,g_1,\dots,f_{2d},g_{2d})$ est telle que $e_i$ est dans $T_{z_i}\CP^1$, $\{f_i,g_i\}$ se projettent sur $H^0(\CP^1,\Nui)$ et $(f_1,\dots,f_{2d})$ se projette sur une base de sections d'un des sous-fibré linéaire de la décomposition du fibré normal, on dit que la base modèle est \emph{ordonnée}.
\end{defi}
\begin{rema}
Quand $\mathcal{B}$ est modèle ordonée, automatiquement $(g_1,\dots,g_{2d})$ est une base pour l'autre sous-fibré linéaire de la décomposition du fibré normal.
\end{rema}

\subsection{Traverser la frontière}\label{trav}

Un point générique $(C^1\cup C^2,\underline{z},u)$ de la frontière est tel que, pour $i$ dans $\{1,2\}$, $C^i$ est isomorphe à $\CP^1$, $u(C^i)$ est une immersion et $\xi=\{C^1 \cap C^2\}$ est un point double ordinaire. Si on note $\deg u(C^1)=d'$ et $\deg u(C^2)=d''$, un élément générique a deux branches à la source avec, disons, $2d'+\ell$ points marqués sur une branche et $2d''-\ell$ sur l'autre (on rappelle que $d=d'+d''$). Par convention $\ell$ est toujours un entier positif.

Soit $\R\D$ une composante pseudo-connexe quelconque de la partie réelle du diviseur irréductible $D(A,B;d',d'')$ où $|A|=2d'+\ell$ avec $1\leq\ell\leq 2d$, ce sont les composantes à étudier d'après la proposition \ref{reducfront}. On distingue deux cas : $d'$ non nul puis $d'$ nul ($d''$ n'est jamais nul puisque $|B|=2d''-\ell<2d''$).
\subsubsection{Chemins transverses. Cas $d'\neq0$}
Un point générique de $\R\D$ possède à la source une structure réelle telle que l'application soit réelle et qui identifie $\RC^1$ et $\RC^2$ avec $\RP^1$.
\paragraph{Choix d'un point de référence.} Prenons un point générique $\mapr$ de $\R\D$ et (re)indexons les points marqués pour que $\{z^{\star}_1,\dots,z^{\star}_{\ell}\} \cup \{z^{\star}_{2d''+1},\dots,z^{\star}_{2d}\}$ soient dans $\R C_{\star}^1$ dans l'ordre cyclique $z^{\star}_1<\dots<z^{\star}_{\ell}<\xi^{\star}<z^{\star}_{2d''+1}<\dots<z^{\star}_{2d}$ et que $\{z^{\star}_{\ell+1},\dots,z^{\star}_{2d''}\}$ (resp. $\emptyset$ si $\ell=2d''$) soient dans $\R C_{\star}^2$ dans l'ordre cyclique $\xi^{\star}<z^{\star}_{\ell+1}<\dots<z^{\star}_{2d''}$ (voir figure \ref{config}).

\begin{figure}[htp]
\begin{center}
\input{configP.pstex_t}
\end{center}
\caption{Au point $\mapr$ : $d=3$, $d'=1$, $\ell=3$}\label{config}
\end{figure}

\paragraph{Choix d'un chemin de référence.}\label{choixchemin}

Soit $B_{\star}$ un voisinage contractile de $\mapr$ dans $\RMc$ de sorte que $\R\D$ découpe $B_{\star}$ en deux composantes connexes. Fixons un chemin générique $\gamma:[0,1] \to B_{\star}$ transverse à $\R\D$ uniquement au point $\mapr$. Sans ambiguïté, on écrit aussi $\gamma$ pour désigner son image $\gamma([0,1]) \subset \RMc$. L'image de $\gamma$ par le morphisme d'évaluation est notée $\gamma_* \subset (\RP^3)^{2d}$ et $\mapt$ désigne le point $\gamma(t)$ pour $t$ dans $[0,1]$. On distingue plus particulièrement $\mapd$ le point de \og départ\fg\ et $\mapa$ le point d'\og arrivée\fg\ du chemin, chacun se situant dans une composante connexe différente de $\RM \cap B_{\star}$. Par généricité de $\gamma$ ces points correspondent à des immersions équilibrées.

\paragraph{Construction d'un chemin de comparaison.}\label{chemintd}

On construit un chemin $\tilde{\gamma}$ associé à $\gamma$ mais à valeurs dans le lieu régulier du morphisme d'évaluation. Au point $\mapr$ est associé le point $\maprtd$ correspondant à la m\^eme courbe mais avec des points marqués différents, de sorte que $\tilde{z}^{\star}_i$ soit dans $\R C^2$ pour $1 \leq i \leq \ell$ et $\tilde{z}^{\star}_i$ soit égal à $z^{\star}_i$ pour $\ell<i\leq2d$. De plus on impose l'ordre cyclique $\xi^{\star}<\tilde{z}^{\star}_1<\dots<\tilde{z}^{\star}_{\ell}<\dots<\tilde{z}^{\star}_{2d''}$ sur $\R C^2$.\label{maprtilde} (Voir figure \ref{configPtd}.)

\begin{figure}[htp]
\begin{center}
\input{configPtd.pstex_t}
\end{center}
\caption{Au point $\maprtd$ : $d=3$, $d'=1$, $\ell=3$}\label{configPtd}
\end{figure}
Par construction $\maprtd$ est dans $\mathcal{K}_{0}^d$ donc, d'après le lemme \ref{imcodim1}, c'est un point régulier du morphisme d'évaluation. Soit $(C_{\star},\underline{z}^{\star}\cup\underline{\tilde{z}}^{\star},u_{\star})$ le point de $\RMcl$ dont la projection par le morphisme d'oubli $\ob^{2d+\ell}_{\ell+1,\dots,2\ell}$ est précisément $(C_{\star},\underline{z}^{\star},u_{\star})$. Ce qui revient à écrire $\underline{z}^{\star}\cup\underline{\tilde{z}}^{\star}$ pour désigner le point $(z^{\star}_1,\dots,z^{\star}_{\ell},\tilde{z}^{\star}_1,\dots,\tilde{z}^{\star}_{\ell},z^{\star}_{\ell+1},\dots,z^{\star}_{2d})$ de $(C_{\star})^{2d+\ell}\setminus Diag_{2d+\ell}$. Soient $\Gamma$ un chemin qui relève $\gamma$ dans $\RMl^*$ pour le morphisme d'oubli $\ob^{2d+\ell}_{\ell+1,\dots,2\ell}$ au point $(C_{\star},\underline{z}^{\star}\cup\underline{\tilde{z}}^{\star},u_{\star})$ et $\tilde{\gamma}:[0,1] \to \RMc^*$ le chemin composé $\ob^{2d+\ell}_{1,\dots,\ell}\circ\Gamma$
\begin{displaymath}
\xymatrix{\Gamma \subset \RMcl^* \ar[r]^{\ob^{2d+\ell}_{1,\dots,\ell}} & \tilde{\gamma} \subset \RMc^*.}
\end{displaymath}
Le chemin de comparaison $\tilde{\gamma}$ est transverse à la frontière au point $\maprtd$ et compatible avec $\gamma$ dans le sens où les courbes et les points marqués en commun coïncident tout le long. Comme précédemment on note $\mapttd=\tilde{\gamma}(t)$, pour $t$ dans $[0,1]$.

\subsubsection{Base modèle positive au point de départ}\label{basd}

Fixons pour la suite une orientation sur l'espace projectif réel $\RP^3$.
\begin{defi}
Soit $(e_1,\dots,e_{k},f_1,g_1,\dots,f_{2d},g_{2d})$ une base modèle ordonnée de $\TRMmc$ où $\map$ est une immersion équilibrée. Alors pour tout triplet de vecteurs $(e_i,f_i,g_i)$ on associe un signe $\angle(e_i,f_i,g_i) = \{(+),(-)\}$ selon que son image par $d|_{\map}\ev$ soit une base positive ou négative de $T_{u(z_i)}\RP^3$.
\end{defi}

On note $A_0=u_0(C_0)$ la courbe réelle définie par le point de départ du chemin de référence $\gamma$. Fixons une orientation $\mathfrak{o}_{\R A_0}$ sur $\R A_0$ et construisons une base positive de $\TRMd$ comme suit. Les $2d$ premiers vecteurs $\{e_i^0 \in T_{z^0_i}C_0\}_{i=1,\dots,2d}$ sont tels que les vecteurs $\{d_{z^0_i}u_0(e_i^0)\}_{i=1,\dots,k_d}$ de $T_{\RA_0}$ soient positifs relativement à $\mathfrak{o}_{\R A_0}$. Les $4d$ derniers vecteurs $\{f_i^0,g_i^0\}_{i=1,\dots,2d}$ complètent la base en une base modèle ordonnée de l'espace tangent $\TRMd$ avec $\angle(e^0_i,f^0_i,g^0_i)=(+)$ pour tout $i$ dans $[2d]$. La base ainsi définie $\mathcal{B}_0=\basd$ est positive. Le fibré normal $\Nud$ est un fibré réel et on note $H^0_{\R}(C,\Nud)$ la partie réelle de $H^0(C,\Nud)$. Par construction, les éléments $\{f_i^0,g_i^0\}_{i=1,\dots,2d}$ de $\mathcal{B}_0$ se projettent dans $H^0_{\R}(C,\Nud)$ et chaque couple $(f_i,g_i)$ sur une base de $H^0_{\R}(C,\Nuid)$.

Suivant la même méthode on construit une base modèle ordonnée positive au point de départ du chemin de comparaison (dans $\TRMdtd$) que l'on note $\tilde{\mathcal{B}}_0=(\tilde{e}_1^0,\dots,\tilde{e}_{2d}^0,\tilde{f}_1^0,\tilde{g}_1^0,\dots,\tilde{f}_{2d}^0,\tilde{f}_{2d}^0)$. Chaque couple $(\tilde{f}_i,\tilde{g}_i)$ se projette sur une base de $H^0_{\R}(\CP^1,\Nuidtd)$ et donc le $4d$-uplet $(\tilde{f}_1^0,\tilde{g}_1^0,\dots,\tilde{f}_{2d}^0,\tilde{g}_{2d}^0)$ est une base de $H^0_{\R}(\CP^1,\Nud)$.

\subsubsection{Homotopies et trivialisation. Cas $d'\neq 0$}

\paragraph{Homotopie au point de départ.}\label{homotoppoint}

\begin{lemm}\label{turnlemmaP}
Le $6d$-uplet $(e_1^0,\dots,e_{2d}^0,\tilde{f}_1^0,\tilde{g}_1^0,\dots,\tilde{f}^0_{2d},\tilde{g}^0_{2d})$ définit une base standard positive de $T_{\gamma(0)}\RMc$.
\end{lemm}

\begin{proof}
Soit le point $\Gamma(0)=(C,\underline{z}^0\cup\underline{\tilde{z}}^0,u_0) \in \RMl^*$ comme défini dans la construction du chemin de comparaison (\ref{chemintd}). Quitte à inverser le chemin, on peut supposer que l'ordre cyclique des points $\underline{z}^0\cup\underline{\tilde{z}}^0$ s'écrive $z^0_1<\dots<z^0_{\ell}<\tilde{z}^0_1<\dots<\tilde{z}^0_{\ell}<z^0_{\ell+1}<\dots<z^0_{2d}$ (voir figure \ref{config0P}).
\begin{figure}[htp]
\begin{center}
\input{config0P.pstex_t}
\end{center}
\caption{Au point $\gamma(0)$ : $d=3$, $\ell=3$}\label{config0P}
\end{figure}\\
Fixons des représentants $(\underline{z}^0,u_0)$ et $(\underline{\tilde{z}}^0,u_0)$ pour $\mapd$ et $\mapdtd$ dans $(\RP^1)^{2d} ×\RMor^d$ de sorte que $z^0_i=\tilde{z}^0_i$ lorsque $i$ est dans $\{\ell+1,\dots,2d\}$. Puisque $u_0$ est une immersion équilibrée, on rappelle que $\Nud$ se décompose en $\OC(2d-1) \oplus \OC(2d-1)$. On considère un isomorphisme de $\R$-espaces vectoriels entre $H^0_{\R}(\CP^1,\OC(2d-1))$ et $R_{2d-1}[X,Y]$, l'espace vectoriel des polynômes homogènes réels de degré $2d-1$. Ainsi $H^0(\CP^1,\Nui)$ est associé au sous-espace de $R_{2d-1}[X,Y]^{\oplus 2}$ des couples de polynômes qui s'annulent en $\{z_1,\dots,z_{2d}\} \setminus \{z_i\}$. D'après l'ordre cyclique sur les points marqués, on peut construire un chemin $c^0:[0,1] \to (\RP^1)^{2d} \setminus Diag_{2d}$ qui relie $c^0(0)=\underline{z}^0$ à $c^0(1)=\underline{\tilde{z}}^0$ et qui reste constant en $z^0_{\ell+1}=\tilde{z}^0_{\ell+1},\dots,z^0_{2d}=\tilde{z}^0_{2d}$. On construit ainsi une homotopie de base $h_{t \in [0,1]}^0$ dans l'espace vectoriel $R_{2d-1}[X,Y]$ de sorte que, pour $i$ dans $[2d]$, le couple $(h_1^0(f_i^0),h_1^0(g_i^0))$ définisse une base du sous-espace $H^0_{\R}(\CP^1,\Nuidtd)$. D'après le choix de $\tilde{\mathcal{B}}_0=(\tilde{e}_1^0,\dots,\tilde{e}_{k_d}^0,\tilde{f}_1^0,\tilde{g}_1^0\dots,\tilde{f}_{2d}^0,\tilde{g}_{2d}^0)$, les vecteurs $\tilde{e}_i$ de $T_{\tilde{z}_i}\RP^1$ sont tous orientés positivement relativement à $\mathfrak{o}_{\RA_0}$. On obtient par $\mathbf{h}^0=h^0\oplus h^0$ une famille de vecteurs de $\TRMd$ ordonnée avec $(\tilde{e}_1^0,\dots,\tilde{e}_{k_d}^0)$ car $\mathbf{h}^0_1(f^0_i,g^0_i)$ est dans $H^0_{\R}(\CP^1,\Nuidtd)$ et positive puisque $\angle(\tilde{e}_i^0,\mathbf{h}^0_1(f^0_i),\mathbf{h}^0_1(g^0_i))=\angle(e_i^0,f^0_i,g^0_i)=(+)$ pour $i$ dans $[2d]$. On en déduit l'existence de réels $\lambda^i_+ > 0$ et $\mu^i_+ > 0$ tels  que $h^0_1(f^0_i)=\lambda^i_+\tilde{f}^0_i$ et $h^0_1(g^0_i)=\mu^i_+\tilde{g}^0_i$ pour $i$ dans $[2d]$. En conclusion, $(f^0_1,g^0_1,\dots,f_{2d}^0,g_{2d}^0)$ et $(\tilde{f}_1^0,\tilde{g}_1^0,\dots,\tilde{f}_{2d}^0,\tilde{g}_{2d}^0)$ sont homotopes comme bases de $H^0_{\R}(\CP^1,\Nud)$ et le résultat s'en déduit par positivité de la base $\mathcal{B}_0$.
\end{proof}

\paragraph{Trivialisation le long du chemin.}\label{trivial}

On décrit une trivialisation du fibré tangent $\TRMc$ le long du chemin de comparaison $\tilde{\gamma}$. Cette dernière nous permet de construire une trivialisation le long du chemin de référence ${\gamma}$. Le chemin de comparaison est dans le lieu régulier du morphisme d'évaluation, en particulier la restriction du morphisme d'évaluation réel à $\tilde{\gamma}$ est de jacobien partout non nul et la décomposition \eqref{decomp} se relève le long de $\tilde{\gamma}$
$$T|_{\tilde{\gamma}}{\RMc}=\bigoplus_{i=1}^{2d}(\Rev)^{*}T|_{\tilde{\gamma}}^{(i)}{\RP^3}.$$
Rappelons que $\hat{ev}_i={\ob_i^{2d}}\circ{\evl}$ et $\ev^*(T^{(i)}_{(\CP^3)^{2d}})=\Ker d\hat{ev}_i$.
Soit $\tilde{\Phi}$ une trivialisation de $T|_{\tilde{\gamma}}{\RMc}$ issue de $\mathcal{B}_0$, compatible avec la décomposition \eqref{decomp} et adaptée à la suite exacte de la proposition \ref{decompfib} (voir \cite{Pu}). On note $\tilde{e}_i:\tilde{\gamma} \to \Ker d\ob_i^{2d}$ les $2d$ sections tautologiques réelles de chaque sous-fibré $\Ker d\ob_i^{2d}|_{\tilde{\gamma}}$ de sorte que  $\tilde{e}_i(\tilde{\gamma}(0))=\tilde{e}_i^0$ et $\tilde{f}_i,\tilde{g}_i:\tilde{\gamma} \to \Ker d\hat{ev}_i/\Ker d\ob_i^{2d}$ les $4d$ sections tautologiques réelles telles que  $\tilde{f}_i(\tilde{\gamma}(0))=\tilde{f}_i^0$,  $\tilde{g}_i(\tilde{\gamma}(0))=\tilde{g}_i^0$ et $(\tilde{f}_i(\tilde{\gamma}(t)),\tilde{g}_i(\tilde{\gamma}(t)))$ forment une base, pour $i$ dans $[2d]$ et $t$ dans $[0,1]$.

Soit $\phi$\label{phi} une trivialisation de $\Ker d\ob$ le long de $\gamma$, issue de $\mathcal{B}_0|_{\ker d\ob}$ et compatible avec la décomposition en droites du lemme \ref{propkero}. On note $e_i:\gamma \to \Ker d\ob_i^{2d}$ les $2d$ sections tautologiques de chaque sous-fibré $\Ker d\ob_i^{2d}|_{\gamma}$ issues de $e^0_i \in \ker d_{\gamma(0)}\ob_i^{2d}$. Le $6d$-uplet $\overline{\mathcal{B}}_t=\bast_{t \in [0,1]}$ est une famille libre de sections du fibré $T|_{\gamma}{\RMc}$ qui définie la trivialisation recherchée.

\paragraph{Revenir au modèle.}\label{B1}

On décrit une homotopie entre la base standard $\mathcal{\overline{B}}_1$ de $\TRMa$ et une base modèle. La méthode est identique à celle de la démonstration du lemme \ref{turnlemmaP}.

\begin{lemm}\label{basa}
Il existe une base modèle $\mathcal{B}_1=(e_1^1,\dots,e_{2d}^1,f_1^1,g_1^1,\dots,f_{2d}^1,g_{2d}^1)$ de $\TRMa$ ayant les propriétés suivantes
\begin{itemize}
\item les $4d$-uplets $(\tilde{f}^1_1,\tilde{g}_1^1,\dots,\tilde{f}^1_{2d},\tilde{g}^1_{2d})$ et $(f^1_1,g^1_1,\dots,f^1_{2d},g^1_{2d})$ se projettent sur des bases homotopes dans $H^0_{\R}(C_1,\Nua)$ ;
\item pour $i$ dans $\{1,\dots,2d''-\ell\}$, le couple $(f_i^1,g_i^1)$ se projette sur une base de $H^0_{\R}(C_1,\mathcal{N}_{u_1,-\ob^{2d}_{\ell+i}(\underline{z})})$ ;
\item pour $i$ dans $\{2d''-\ell+1,\dots,2d''\}$, le couple $(f_i^1,g_i^1)$ se projette sur une base de $H^0_{\R}(C_1,\mathcal{N}_{u_1,-\ob^{2d}_{2d''+1-i}(\underline{z})})$ ;
\item pour $i$ dans $\{2d''+1,\dots,2d\}$, le couple $(f_i^1,g_i^1)$ se projette sur une base de $H^0_{\R}(C_1,\Nui)$.
\end{itemize}
On convient lorsque $2d''-\ell$ est nul de remplacer $\{1,\dots,2d''-\ell\}$ par l'ensemble vide.
\end{lemm}

\begin {proof}
Quitte à restreindre $\gamma$, l'ordre cyclique des points $\underline{z}^1 \cup \underline{\tilde{z}}^1$ est $z^1_1<\dots<z^1_{\ell}<z^1_{2d''}<\dots<z^1_{\ell+1}<\tilde{z}^1_{\ell}<\cdots<\tilde{z}^1_{1}<z^1_{2d''+1}<\dots<z^1_{2d}$ (voir figure \ref{config1}).
\begin{figure}[htp]
\begin{center}
\input{config1P.pstex_t}
\end{center}
\caption{Au point $\gamma(1)$ : $d=3$, $\ell=3$}\label{config1}
\end{figure}
On choisit des représentants $(z^1,u_1)$ et $(\tilde{z}^1,u_1)$ dans $(\RP^1)^{2d}×\RMor^d $ respectivement de $\mapa$ et $\mapatd$ de sorte que $z^1_i=\tilde{z}^1_i$ pour $i$ de $\ell+1$ à $2d$. D'après l'ordre cyclique sur les points $\underline{z}^1$ et $\underline{\tilde{z}}^1$, on peut construire un chemin $c^1$ dans $(\RP^1)^{2d} \setminus Diag_{2d}$ qui relie $\underline{\tilde{z}}^1$ à $(z^1_{\ell+1},\dots,z^1_{2d''},z_{\ell}^1,\dots,z_{1}^1,z_{2d''+1}^1,\dots,z_{2d}^1)$ (voir figure \ref{chap4}).
\begin{figure}[htp]
\begin{center}
\input{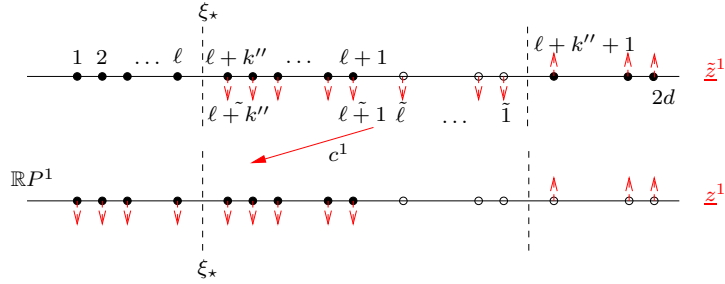}
\end{center}
\caption{Description de $c^1$ et $\underline{z}^1\cup \underline{\tilde{z}}^1 \in (\RP^1)^{2d+\ell}$}\label{chap4}
\end{figure} Puisque $u_1$ est une immersion équilibrée, le fibré $\Nua$ est isomorphe au fibré $\OC(2d-1) \oplus \OC(2d-1)$. De plus, d'après la construction, les éléments $\tilde{f}^1_i$ et $\tilde{g}^1_i$ sont dans $H^0_{\R}(\CP^1,\OC(-\widehat{\tilde{z}^1_{i}}))$ où $\widehat{\tilde{z}^1_{i}}$ désigne la collection de points marqués $\underline{\tilde{z}}^1$ dont on a ôté le $i$-ème. Par suite, on définit comme en \ref{homotoppoint} une homotopie de base, notée ${\bf h}^1_{t \in [0,1]}$, dans $H^0_{\R}(\CP^1,\OC(2d-1)^{\oplus2})$ telle que ${\bf h}^1_0=id$ et
\begin{itemize}
\item ${\bf h}^1_1(\tilde{f}^1_i,\tilde{g}^1_i) \in H^0_{\R}(\CP^1,\OC(-\widehat{z^1_{\ell+i}})^{\oplus2}))$  pour $i \in \{1,\dots,2d''-\ell\}$ ;
\item ${\bf h}^1_1(\tilde{f}^1_i,\tilde{g}^1_i)\in H^0_{\R}(\CP^1,\OC(-\widehat{z^1_{2d''+1-i}})^{\oplus2})$ pour $i \in \{2d''-\ell+1,\dots,2d''\}$ ;
\item ${\bf h}^1_1(\tilde{f}^1_i,\tilde{g}^1_i) \in  H^0_{\R}(\CP^1,\OC(-\widehat{z^1_{i}})^{\oplus2})$ pour $i \in \{2d''+1,\dots,2d\}$.
\end{itemize}
On pose $(f^1_i,g^1_i)= {\bf h}^1_1(\tilde{f}^1_i,\tilde{g}^1_i)$ pour $i$ dans $[2d]$ et on obtient le résultat.
\end{proof}

\subsubsection{Cas $d'=0$}\label{degal0}\label{noneff}

La situation où $\mapr$ appartient à $D(A,B,0,d)$ nécessite une étude différenciée. En effet, l'application $u$ contracte alors une branche de $C$ et le lemme \ref{propkero} comme les propositions \ref{decomptan} et \ref{decompfib} ne sont plus valables en ce point.
\paragraph{Chemins.}
Soit $\R \mathcal{D}_{0,{\ell}}^d$ une composante pseudo-connexe de $\R\mathcal{K}_{0,{\ell}}^d$ qui soit de codimension un dans $\RMc$ (nécessairement ${\ell}$ est au moins deux par stabilité.) On choisit un point générique $\mapr$ dans $B_{\star} \cap \K$ tel que, quitte à changer l'indexation, les points $\{z^{\star}_1,\dots,z^{\star}_{{\ell}}\}$ sont sur la branche de classe nulle et l'ordre cyclique sur la branche de classe non nulle est $z^{\star}_{{\ell}+1}<\dots<z^{\star}_{2d}$. On définit un chemin de référence $\gamma : t \in [0,1] \mapsto (C_t,z^t,u_t) \in \RMc^*$ transverse à $\K$ en $\mapr$ au paramètre $t_{\star}$ mais dans le \og sens\fg\ des applications. C'est-à-dire que seuls les points marqués dépendent du paramètre $t$, l'application $u_t$ fixant l'image $u_0(\CP^1)=u_{t}(\CP^1)=u_{\star}(C_{\star})=u_1(\CP^1)$ pour tout $t$ dans $[0,1]$. De plus, pour simplifier on exige que $\gamma(0)=\gamma(1)$ après renversement dans l'ordre cyclique $u_0(z^0_1)=u_1(z^1_1),\dots,u_0(z^0_{{\ell}})=u_1(z^1_{{\ell}})$. De cette façon, $\mapa$ peut s'écrire $(\CP^1,z^0_{{\ell}},\dots,z^0_1,z^0_{{\ell}+1},\dots,z^0_{2d},u_0)$. Par généricité, l'application $u_{\star}$ peut être choisie équilibrée et par suite $u_0=u_1=u_{\star}|_{C_{\star}^2}$ sont toutes équilibrées.

\paragraph{Bases.}
On construit une base modèle ordonnée de $\TRMd$ en se donnant $4d$ éléments $\{f^0_i,g^0_i\}_{i=1,\dots,2d}$ de sorte que chaque couple $(f_i,g_i)$ se projette sur une base de $H_{\R}^0(\CP^1,\Nuid)$ qui respecte la décomposition du fibré normal $\Nud = \OC(2d-1) \oplus \OC(2d-1)$. Par le même procédé qu'en \ref{basd} on construit une base modèle ordonnée positive $\mathcal{B}_0=(e_1,\dots,e_{2d},f_1,g_1,\dots,f_{2d},g_{2d})$ d'après le choix d'une orientation $\mathfrak{o}_{\RA_0}$ sur la partie réelle de $A_0=u_{0}(\CP^1)$.

\paragraph{Trivialisation.}\label{orient0}
On construit une trivialisation de $T|_{\gamma}{\RMc}$ issue de $\mathcal{B}_0$ qui soit tautologique sur les $4d$ derniers vecteurs. C'est-à-dire que le long du chemin on trivialise le sous-fibré défini par les vecteurs $(f_1,g_1,\dots,f_{2d},g_{2d})$ de la base modèle $\mathcal{B}_0$ en se donnant $4d$ sections constantes $t \in [0,1] \mapsto f^0_i,g^0_i \in H^0(\CP^1,u_0^*T_{\CP^3})$ vu que l'application $u_t=u_0$ est constante. Pour la partie de $T|_{\gamma}{\RMc}$ engendrée par $\Ker d\ob$ on ne peut pas définir les vecteurs $e_i^t$ comme dans le cas $d'$ non nul car $\ker d|_{\mapr}\ob \cap \ker d|_{\mapr}\Rev \neq \{0\}$. Puisque l'application $u_0$ est constante le long du chemin, on peut décrire la courbe universelle au-dessus de $\gamma$ comme l'espace des déformations de la source. Considérons le produit $[0,1] ×\CP^1 ×\{u_0\}$ et les $2d$ sections $t \in [0,1] \mapsto z_i^t \in \CP^1$ définies par le chemin $\gamma$. La courbe universelle $\mathcal{U}^d_{2d}(\CP^3)$ au-dessus de $\Mc^*$ restreinte à ${\gamma}$ est l'éclaté $\Bl_{\xi}([0,1] ×\CP^1)$ du produit $[0,1] \times \CP^1$ en l'unique point de concours $(t_{\star},\xi)$ des ${\ell}$ premières sections (voir \cite{F-P}).
\begin{figure}[htp]
\begin{center}
\input{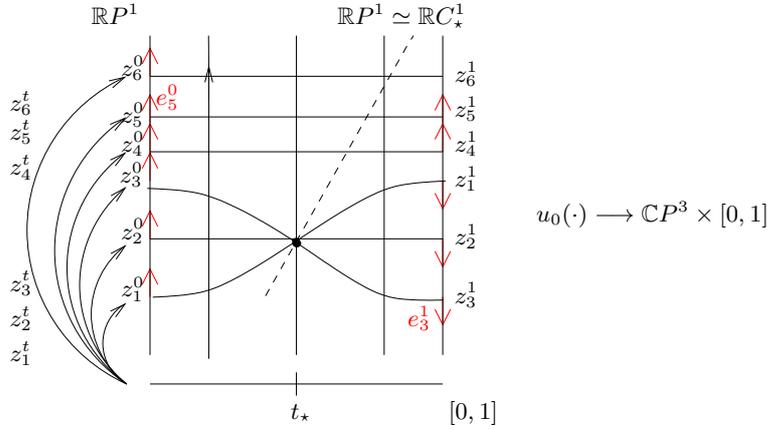}
\end{center}
\caption{Description de la courbe universelle $\mathcal{U}^d_{2d}(\CP^3)|_{\gamma}$ ($d'=0$, ${\ell}=3$)}\label{Ugamma}
\end{figure}\\
On définit ainsi une trivialisation de base $(e_1^t,\dots,e_{2d}^t)_{t \in [0,1]} \subset \Ker d\ob$ dans $\Bl_{\xi}([0,1] \times \CP^1)$ issue de $(e_1,\dots,e_{2d})_{t=0}$. Le résultat de cette construction est une base modèle non ordonnée $\mathcal{B}_1$ telle qu'on peut la décrire par la donnée $\mathcal{B}_1 \equiv (-e_{{\ell}},\dots,-e_1,e_{{\ell}+1},\dots,e_{2d},f_1,g_1,\dots,f_{2d},g_{2d})$ (voir figure \ref{Ugamma}).

\subsubsection{Évaluation de bases}\label{evalbase}
On compare les orientations induites par les bases  $\mathcal{B}_0$ et $\mathcal{B}_1$ dans une trivialisation du fibré tangent à l'image du chemin de référence par le morphisme d'évaluation. Lorsque ces orientations coïncident, on en déduit que $\epsilon(\R\D)=0$, dans le cas contraire on obtient $\epsilon(\R\D)=1$.

\paragraph{Trivialisation à l'image. Cas $d'\neq0$.}\label{trivRX}
On se place dans une trivialisation de $T_{(\RP^3)^{2d}}$ le long de $\gamma_*$ qui soit compatible avec le morphisme d'évaluation. On projette la base modèle ordonnée positive du paragraphe \ref{basd} sur une base en $\gamma_*(0)$. Reprenons la trivialisation $\phi$ définie en \ref{phi} ayant pour sections tautologiques $e_i : \gamma \subset \RMc \to \Ker d\ob_i^{2d}|_{\gamma}$ pour $i$ dans $[2d]$. Puisque $d\Rev$ est injective sur $\Ker d\ob|_{\gamma}$, on construit une famille libre de $2d\ (=\dim \Ker d\ob)$ sections $v_i: \gamma_* \subset (\RP^3)^{2d} \to T|_{\gamma_*}(\RP^3)^{2d}$ comme images des sections $e_i$. C'est-à-dire que si on écrit ${v}_i^{t}$ pour $v_i(\gamma_*(t))$ alors ${v}_i^t=d|_{\gamma_*(t)}\Rev(e_i^t) \in T^{(i)}_{\gamma_*(t)}\RP^3$ pour $t$ dans $[0,1]$. On complète la base en choisissant pour chaque $v_i^0$ un couple de vecteurs $(w_i^0,\omega_i^0)$ dans $T^{(i)}_{\gamma_*(0)}\RP^3$ de sorte que $(\overline{w}_i^0,\overline{\omega}_i^0)=d|_{\gamma_*(0)}\Rev(f_i^0,g_i^0)$ où $\overline{w}_i^0$ représente la classe au quotient de $w_i^0$ dans la suite exacte $$\xymatrix{0 \ar[r]& <v^0_i>_{\R} \ar[r]& T^{(i)}_{\gamma_*(0)}\RP^3 \ar[r]& T^{(i)}_{\gamma_*(0)}{\RP^3}/<v^0_i>_{\R} \ar[r]& 0.}$$ Ainsi chaque triplet $(v_i^0,w_i^0,\omega_i^0)$ est une base positive de $T_{\gamma_*(0)}\RP^3_{(i)}$. On considère une trivialisation de $T|^{(i)}_{{\gamma}_*}\RP^3$ et on complète le $2d$-uplet $(v^t_1,\dots,v^t_{2d})_{t \in [0,1]}$ par $4d$ sections $\{w_i^t,\omega_i^t ; t \in [0,1]\}_{i \in [2d]}$ issues de $(w_i^0,\omega_i^0)$. On obtient une section de bases positives $\mathcal{V}_t=({v}_1^t,\dots,{v}_{2d}^t,{w}_1^t,{\omega}_1^t,\dots,{w}_{2d}^t,{\omega}_{2d}^t)$ de $T_{(\RP^3)^{2d}}$ le long du chemin $\gamma_*$. De plus $\mathcal{V}_0$ est égal à $d|_{\gamma(0)}\Rev(\mathcal{B}_0)$ et $(d\Rev)^{-1}(\mathcal{V}_1)$ est une base modèle ordonnée positive de $\TRMa$.

\paragraph{Trivialisation à l'image. Cas $d'=0$.}\label{trivX0}

On prend la trivialisation la plus évidente de $T_{(\RP^3)^{2d}}$ le long du chemin $\gamma_*$ qui soit issue de $d|_{\gamma(1)}\Rev(\mathcal{B}_0)$ et qui respecte l'inclusion $T_{u(\underline{z})}\R u_0(\CP^1) \subset T_{u(\underline{z})}\RP^3$. C'est-à-dire une section de bases $\mathcal{V}_t=({v}_1^t,\dots,{v}_{2d}^t,{w}_1^t,{\omega}_1^t,\dots,{w}_{2d}^t,{\omega}_{2d}^t)$ telle que $\mathcal{V}_0$ soit égale à l'image de $d|_{\gamma(0)}\Rev(\mathcal{B}_0)$ et $\mathcal{V}_1$ soit équivalente à la base de $T_{\gamma_*(1)}\RP^3$ donnée par $({v}^0_{\ell},\dots,{v}^0_1,{v}^0_{\ell+1}, \dots {v}_{2d}^0,{w}_1^0,{\omega}_1^0,\dots,{w}_{2d}^0,{\omega}_{2d}^0)$.

\subsubsection{Matrices de transition}\label{mattrans}

On décrit la base $d|_{\gamma(1)}\Rev(\mathcal{B}_1)$ dans la base $\mathcal{V}_1$ sous forme matricielle. Pour simplifier l'expression et le calcul du déterminant on se limite à définir une matrice semblable. La proposition \ref{matnonnul} étudie le cas $d'$ non nul et la proposition \ref{matnul} le cas où $d'$ est nul.

\begin{prop}\label{matnonnul}
Suivant les trivialisations définies en \ref{trivial} et \ref{trivRX}, l'expression matricielle de la base $d_{\gamma(1)}\Rev(\mathcal{B}_1)$ dans la base $\mathcal{V}_1$ est semblable à la matrice\[\mathcal{B}_1|_{\mathcal{V}_1}=\left(\begin{array}{c|c|c|c|c}-I_{\ell} & 0 & 0 & 0 & 0\\ \hline 0 & I_{2d-\ell} & 0 & 0 & 0 \\ \hline 0 & 0 & 0 & \mathcal{J}_{\ell} & 0\\ \hline 0 & 0 & I_{2(2d''-\ell)} & 0 & 0 \\ \hline 0 & 0 & 0 & 0 & I_{2(2d')}\end{array}\right)\]où $\mathcal{J}:=\left(\begin{array}{c|c|c|c}0 & \cdots & 0 & I_2 \\ \hline 0 & \cdots & I_2 & 0\\ \hline \vdots & {\mathinner{\mkern2mu\raise1pt\hbox{.}\mkern2mu\raise4pt\hbox{.}\mkern2mu\raise7pt\hbox{.}\mkern1mu}} & \vdots & \vdots \\ \hline I_2 & \cdots & 0 & 0\end{array}\right)$ la matrice de permutation cyclique par paires.
\end{prop}
\begin{lemm}\label{signarrivee}
La base modèle $\mathcal{B}_1=\basa$ de l'espace tangent $\TRMa$ définie en \ref{B1} vérifie
\begin{enumerate}
\item $\angle(e^1_i,f^1_{2d''-+1-i},g^1_{2d''+1-i})=(-)$ pour $1\leq i \leq \ell$ ;
\item $\angle(e_{i}^1,f_{i-\ell}^1,g_{i-\ell}^1)=(+)$ pour $\ell<i\leq2d''$ ;
\item $\angle(e_{i}^1,f_{i}^1,g_{i}^1)=(+)$ pour $2d''<i\leq 2d$.
\end{enumerate}
\end{lemm}

\begin{proof}
On fixe une orientation $\mathfrak{o}_{\R A_1}$ sur la partie réelle de $A_1=u_1(\CP^1)$ de sorte que les vecteurs $\{e^1_i ; i=\ell+1,\dots,2d''\}$ (c.-à-d. les éléments de la base issues de $T_{C_{\star}^2}$) soient positifs. Les vecteurs $\{e^1_i ; i=1,\dots,\ell\}$ et $\{\tilde{e}^1_i ; i=2d''+1,\dots,2d\}$ sont alors orientés négativement relativement à $\mathfrak{o}_{\R A_1}$ puisque le chemin $\gamma$ est transverse à la frontière. De même, en $\tilde{\gamma}(1)$ les vecteurs $\{\tilde{e}^1_i ; i=1,\dots,2d''\}$ sont orientés positivement relativement à $\mathfrak{o}_{\R A_1}$ alors que les vecteurs $\{\tilde{e}^1_i ; i=2d''+1,\dots,2d\}$ le sont négativement d'après le lemme \ref{maprtilde}. Reprenons la construction de $\mathcal{B}_1$ et considérons le relevé du chemin $c^1$ dans $(\RP^1)^{2d} × \RMor^d$ et l'homotopie de base $h^1$ de $H^0_{\R}(\CP^1,\OC(2d-1))$ associée par le lemme \ref{basa}. Une trivialisation de $T|_{c^1}\RM$ issue de $\tilde{\mathcal{B}}_1$ compatible avec la suite exacte (\ref{decomp}) et définie par cette homotopie fournit une base de $\TRMa$ dont les $4d$ derniers vecteurs sont $(f^1_1,g^1_1,\dots,f^1_{2d},g^1_{2d})$. On obtient le long de la première composante une trivialisation $(\tilde{e}_1^1(t),\dots,\tilde{e}_{2d}^1(t))_{t \in [0,1]}$ de $T|_{c^1}(\RP^1)^{2d}$ qui vérifie $\angle(\tilde{e}^1_i(1),f^1_i,g^1_i)=(+)$ car $\angle(\tilde{e}_i^1(0),\tilde{f}^1_i,\tilde{g}^1_i) = (+)$, pour $i$ dans $[2d]$. Cependant, pour $i$ dans $[\ell]$, $e^1_i$ étant orienté négativement relativement à $\mathfrak{o}_{\R A_1}$, il est opposé à l'orientation définie par $\tilde{e}_{i}^1(0)$. Or, cette orientation est aussi celle définie par $\tilde{e}^1_{2d''+1-i}(1) $ car $2d''+1-i$ est dans $[2d'']$. D'autre part, pour $i$ dans $\{\ell+1,\dots,2d''\}$, $e^1_i$ est orienté positivement relativement à $\mathfrak{o}_{\R A_1}$ et donc à l'orientation définie par $\tilde{e}_{i}^1(0)$ qui est aussi celle définie par $\tilde{e}_{i-\ell}^1(0)$ car $i-\ell$ est dans $[2d'']$. Enfin, pour $i$ dans $\{2d''+1,\dots,2d\}$, $e^1_i$ est orienté positivement relativement à $\mathfrak{o}_{\R A_1}$ et donc relativement à $\tilde{e}^1_i(1)$. Finalement, $\tilde{e}^1_i(1)= \lambda_+^i e_i^1$  si $i$ est dans $\{2d'',\dots,2d\}$ et $\tilde{e}^1_{i-\ell}(1)= \lambda_+^i e_i^1$  si $i$ appartient à $\{\ell+1,\dots,2d''\}$ où $\lambda_+^i \in \R_+^*$ alors que $\tilde{e}^1_{2d''+1-i}(1)=\lambda_-^i e_i^1$ avec $\lambda_-^i \in \R_-^*$ lorsque $i$ appartient à $[\ell]$. En conséquence, $(f^1_{2d''+1-i}(z^1_i),g^1_{2d''+1-i}(z^1_i))=\Lambda_-^i( \tilde{f}^1_i(\tilde{z}^1_i),\tilde{g}^1_i(\tilde{z}^1_i))$ pour $i$ dans $[\ell]$ et $(f^1_i(z^1_{i}),g^1_i(z^1_{i}))=\Lambda_+^i (\tilde{f}^1_i(\tilde{z}^1_i),\tilde{g}^1_i(\tilde{z}^1_i))$ pour $i$ dans $\{2d'',\dots,2d\}$, de même $(f^1_{i-\ell}(z^1_{i}),g^1_{i-\ell}(z^1_{i}))=\Lambda_+^i (\tilde{f}^1_i(\tilde{z}^1_i),\tilde{f}^1_i(\tilde{z}^1_i))$ pour $i$ dans $\{\ell+1,\dots,2d''\}$, avec $\det \Lambda_-^i<0$ et $\det \Lambda_+^i>0$, d'où le résultat.
\end{proof}
\begin{proof}[Démonstration de la proposition \ref{matnonnul}]
D'après le lemme \ref{signarrivee}, la base modèle $\mathcal{B}_1$ est équivalente à la donnée du $6d$-uplet suivant \begin{equation*}\begin{aligned}(e_1,\dots,e_{2d},-f_{2d''},g_{2d''},\dots&,-f_{1+2d''-\ell},g_{1+2d''-\ell},\\&f_1,g_1,\dots,f_{2d''-\ell},g_{2d''-\ell},f_{2d''+1},g_{2d''+1},\dots,f_{2d},g_{2d})\end{aligned}\end{equation*} où $(e_1,\dots,e_{2d},f_1,g_1,\dots,f_{2d},g_{2d})$ désigne une base modèle ordonnée positive de $T_{\gamma(1)}\RMc$ qui se projette sur $\mathcal{V}_1$ (voir \ref{trivRX}). Les premiers vecteurs de base $(e^1_1,\dots,e^1_{2d})$ étant tous positifs relativement à $\mathfrak{o}_{\RA_1}$ on peut décider qu'un des deux vecteurs $(f^1_i,g^1_i)$ pour $i$ dans $\{2d''-\ell+1,\dots,2d''\}$ doit porter un signe négatif et on choisit arbitrairement le premier. On en déduit l'expression d'une matrice semblable à la matrice de transition qui définit l'image de la base $\mathcal{B}_1$ par le morphisme d'évaluation réel dans la base $\mathcal{V}_1$. Pour cela, on remplace $\mathcal{B}_1$ par la donnée \begin{equation*}\begin{aligned}(-e_1,\dots,-e_{\ell},e_{\ell+1},\dots,e_{2d},&f_{2d''},g_{2d''},\dots ,f_{1+2d''-\ell},g_{1+2d''-\ell}, \\&f_1,g_1,\dots,f_{2d''-\ell},g_{2d''-\ell},f_{2d''+1},g_{2d''+1},\dots,f_{2d},g_{2d})\end{aligned}\end{equation*} qui lui est encore équivalente. Vu que $2d-2d''=2d'$ on obtient l'expression matricielle recherchée.
\end{proof}

\begin{prop}\label{matnul}
Cas $d'= 0$, suivant les trivialisations définies en \ref{orient0} et \ref{trivX0}, l'expression de la base $d_{\gamma(1)}\Rev(\mathcal{B}_1)$ dans la base $\mathcal{V}_1$ est semblable à la matrice \[\mathcal{B}_1|_{\mathcal{V}_1}=\left(\begin{array}{c|c|c|c}-I_{\ell} & 0 & 0 & 0 \\ \hline 0 & I_{2d-\ell} & 0 & 0\\ \hline 0 & 0 & \mathcal{J}_{\ell} & 0 \\ \hline 0 & 0 & 0 & I_{2(2d-\ell)}\end{array}\right).\]
\end{prop}
\begin{proof}
La base modèle $\mathcal{B}_1$ est le résultat d'une trivialisation qui est tautologique sur les $4d$ derniers vecteurs. De plus, chaque section $e^t_i$ de l'espace tangent à la source induit une orientation opposée à celle définie par $\mathfrak{o}_{\RA}$ pour $t>t_{\star}$ (en particulier en $t=1$) dès que $i$ est dans $[\ell]$ et sont tautologiques sinon. On peut donc décrire $\mathcal{B}_0$ comme une base modèle par la donnée $(-e_{\ell},\dots,-e_1,e_{\ell+1},\dots,e_{2d},f_1,g_1,\dots,f_{2d},g_{2d})$, d'après le choix du chemin. Cependant, dans cette description, les couples de vecteurs $(f_i,g_i)$ forment une base de $H^0_{c}(\CP^1,\mathcal{N}_{u,-\ob_{\ell+1-i}^{2d}(z)})$ (resp. de $H^0_{c}(\CP^1,\mathcal{N}_{u,-\ob_{i}^{2d}(z)})$) pour $i$ dans $[\ell]$ (resp. pour $i$ dans $\{\ell+1,\dots,2d\}$) d'après le choix du chemin. La base n'est donc pas ordonnée et la matrice de transition est présentable dans la forme annoncée.
\end{proof}

\subsection{Conclusion}
On rappelle que $\R\D$ est une composante pseudo-connexe de la partie réelle du diviseur irréductible de la frontière $D(A,B,d',d'')$ tel que $|A|=2d'+\ell$.

\begin{prop}\label{propfin}Soit $d$ un entier supérieur à $3$. Une composante pseudo-connexe $\R \D$ de la partie réelle de la frontière  participe à représenter (dualement) la première classe de Stiefel-Whitney de $\RMc$ si et seulement si $\ell \equiv 0 \mod (2)$.
\end{prop}

\begin{proof}
C'est le résultat du calcul du déterminant des matrices de transitions définies dans les propositions précédentes. Pour le cas $d'=0$ il est immédiat que ce déterminant vaut $(-1)^{\ell}$ car $\det \mathcal{J}=1$. Dans le cas où $d'$ n'est pas nul, on observe que $$\det \left(\begin{array}{c|c|c|c} ±I & 0 & 0 & 0 \\ \hline 0 & 0 & \mathcal{J} & 0\\ \hline 0 & I_{2(2d''-\ell)} & 0 & 0 \\ \hline 0 & 0 & 0 & I \end{array}\right)=\det \left(\begin{array}{c|c|c|c} ±I & 0 & 0 & 0 \\ \hline 0 & \mathcal{J} & 0 & 0\\ \hline 0 & 0 & I_{2(2d''-\ell)} & 0 \\ \hline 0 & 0 & 0 & I \end{array}\right)$$ puisque l'ordre de la permutation des colonnes associée est toujours pair. Donc le déterminant de la matrice de la proposition \ref{matnonnul} est $(-1)^{\ell}$.
\end{proof}

La proposition \ref{propfin} achève la démonstration du théorème \ref{theo}.

\end{document}